\documentclass{amsproc}
\usepackage{euscript}
\usepackage{cases}
\usepackage{mathrsfs}
\usepackage{bbm}
\usepackage{amssymb}
\usepackage{amsfonts,amsmath,amsxtra,mathdots,mathabx,mathtools}
\usepackage[mathlines]{lineno}
\usepackage{color}
\usepackage{hyperref}
\usepackage{tikz}
\usepackage{appendix,upgreek}

\allowdisplaybreaks

\DeclareFontFamily{U}{matha}{\hyphenchar\font45}
\DeclareFontShape{U}{matha}{m}{n}{
	<5> <6> <7> <8> <9> <10> gen * matha
	<10.95> matha10 <12> <14.4> <17.28> <20.74> <24.88> matha12
}{}
\DeclareSymbolFont{matha}{U}{matha}{m}{n}

\DeclareMathSymbol{\Lt}{3}{matha}{"CE}
\DeclareMathSymbol{\Gt}{3}{matha}{"CF}

\DeclareSymbolFont{mathc}{OML}{txmi}{m}{it}
\DeclareMathSymbol{\varuu}{\mathord}{mathc}{117}
\DeclareMathSymbol{\varvv}{\mathord}{mathc}{118}
\DeclareMathSymbol{\varww}{\mathord}{mathc}{119}

\def\SO{\text{\raisebox{- 2 \depth}{\scalebox{1.1}{$ \text{\usefont{U}{BOONDOX-calo}{m}{n}O} \hspace{0.5pt} $}}}}

\def\valpha{\text{\scalebox{0.86}[1.02]{$\alpha$}}}   
\def\vepsilon{\upvarepsilon}
\def\vnu{\text{{\scalebox{0.86}[1]{$\nu$}}}}

\def\uppii{\text{\scalebox{0.8}[0.96]{$\uppi$}}}
\def\upvchi{\text{\scalebox{0.8}[1]{$\upchi$}}}

\def\viint{	\int \!\!\! \int} 
 
\def\shskip{\hspace{0.5pt}} 

\renewcommand{\Re}{{\mathrm{Re} }}

\def\mod{\mathrm{mod}\,  }
\def\nd{\mathrm{d}}
\def\Tr{\mathrm{Tr}}

\newcommand{\BC}{{\mathbb {C}}}
\newcommand{\BH}{{\mathbb {H}}}
\newcommand{\BQ}{{\mathbb {Q}}} 
\newcommand{\BR}{{\mathbb {R}}} 
\newcommand{\BZ}{{\mathbb {Z}}}

\newcommand{\RC}{{\mathrm {C}}} 
\newcommand{\RN}{{\mathrm {N}}}

\newcommand{\GL}{{\mathrm {GL}}}

\newcommand{\ra}{\rightarrow}

\def\mod{\mathrm{mod}\,  }
\def\nd{\mathrm{d}}

\def\lp {\left (}
\def\rp {\right )}

\def\lb {\left [ }
\def\rb {\right ]}

\def\shskip{\hspace{0.5pt}}

\def\eqq {q}

\newcommand{\CL}{{\text{\large $\EuScript{L}$}}}

\newcommand{\delete}[1]{}

\theoremstyle{plain}

\newtheorem{cor}{Corollary}[section]
\newtheorem{lem}{Lemma}[section]
\newtheorem{theorem}{Theorem}[section] 
\newtheorem{proposition}{Proposition}[section]

\newtheorem*{thm*}{Theorem}
\newtheorem{hyp}{Hypothesis}[section]

\theoremstyle{remark} 
\newtheorem{remark}{Remark}[section] 

\numberwithin{equation}{section}

\usepackage[mathlines]{lineno}

\begin{document}

	\title[Hybrid Weyl Subconvexity over Imaginary Quadratic Fields]{Hybrid Weyl Subconvexity over Imaginary Quadratic Fields}
	\author[Z. Gao, C. Li, and Z. Qi]{Zhengxiao Gao, Changlin Li,  and Zhi Qi}
	\address{School of Mathematical Sciences, Zhejiang University, Hangzhou, 310027, China}
	\email{zxgao@zju.edu.cn,12135012@zju.edu.cn, zhi.qi@zju.edu.cn}

	\begin{abstract}
		On imaginary quadratic fields, we establish the  $\mathrm{GL}_2$ Bessel $\delta$-method and prove the hybrid Weyl-type subconvexity bound for $\mathrm{GL}_2 \times \mathrm{GL}_1$ twisted $L$-functions in the Archimedean aspect of the Hecke character on  $\mathrm{GL}_1$.  
	\end{abstract}

	\thanks{The third author was supported by National Key R\&D Program of China No. 2022YFA1005300.}
	\maketitle

\section{Introduction}

	Let $F $ be an imaginary quadratic field of class number $h_{F} = 1$. Let $\SO$ denote its ring of integers and $\SO^{\times}$ its group of units. Define $w_F = |\SO^{\times}|$. 
Let $\uppii $ be a fixed cuspidal automorphic representation for $ \mathrm{PGL}_2(\SO) \backslash \mathrm{PGL}_2(\BC) $. For $ (t, m) \in \BR \times w_{F} \cdot \BZ $, define the Gr\"ossencharakter $\upvchi_{it, m} $ on $ \SO \smallsetminus \{ 0 \} $ to be 
\begin{align*}
	\upvchi_{it, m} (n) = |n|^{it} (n / |n|)^{m}. 
\end{align*}
Consider the $L$-function $L(s,\uppii \times \upvchi_{it,m})$  associated with $\uppii$ twisted by  $\upvchi_{it,m}$.  The main result  of this paper is the hybrid Weyl-type subconvexity bound in both the $t$- and $m$-aspects as follows.

	\begin{theorem}\label{1thm: Weyl bound}
	We have 
	\begin{equation*}
		L(1/2,\uppii\times \upvchi_{it,m}) \Lt_{\uppii,\vepsilon, F} (1+|t|+|m|)^{2/3+\vepsilon},
	\end{equation*}
where the implied constant depends on $\uppii$, $\vepsilon$, and the field $F$. 
\end{theorem}

The first subconvexity results over an imaginary quadratic field are due to Petridis and Sarnak in 2001 \cite{Sarnak-P-QUE-H3}, in which, for $F = \BQ (i)$ and spherical $\uppii$, it is proven (in our notation) that 
\begin{align*}
	L (1/2, \uppii \times \upvchi_{it, m}) & \Lt_{\uppii,\vepsilon,m} (1+|t|)^{159/166+\vepsilon}, \\
	L (1/2, \uppii \times \upvchi_{it, m}) & \Lt_{\uppii,\vepsilon,t} (1+|m|)^{159/166+\vepsilon}.
\end{align*}

Since the settlement of the subconvexity problem for $\GL_2$ in the seminal work of Michel and Venkatesh in 2010  \cite{Michel-Venkatesh-GL2}, there have been strong subconvexity bounds for $ \GL_2 \times \GL_1$ proven over   general number fields $F$.  In 2014, for a cuspidal automorphic representation $\uppii$ and a Hecke character $\upvchi$,  Han Wu \cite{WuHan-GL2} (see also \cite{Blomer-Harcos-TR,Maga-Sub}) established the Burgess-like subconvexity bound for the twisted $L$-functions:
\begin{equation*}
	L(1/2,\uppii\times\upvchi)\Lt_{\uppii,\vepsilon, F} \mathrm{C} (\upvchi)^{1/2-(1-2\theta)/8+\vepsilon},
\end{equation*}
where $\mathrm{C} (\upvchi)$ is the analytic conductor of $\upvchi$ and $\theta$ is any exponent towards the Ramanujan--Petersson conjecture. 
Recently, Liyang Yang \cite{Yang-Burgess-Bound} proved the genuine Burgess-type bound:
\begin{equation*}
	L(1/2,\uppii\times\upvchi)\Lt_{\uppii,\vepsilon, F} \mathrm{C} (\upvchi)^{3/8+\vepsilon}.
\end{equation*}
Thus, over imaginary quadratic fields $F$, our result improves the Burgess exponent $3/8$ into the Weyl exponent $1/3$  in the case that both $\uppii$ and $ \upvchi $ are unramified at non-Archimedean places (note here that $\mathrm{C}(\upvchi_{it, m}) \asymp (1+|t|+|m|)^2$). 

In the classical case, for a fixed holomorphic modular form $f (z)$ for $\mathrm{SL}_2(\BZ)$, Good \cite{Good-Weyl} proved in 1982 the Weyl-type subconvexity bound in the $t$-aspect:
\begin{equation*}
	L(1/2+it,f) \Lt_{f,\vepsilon} (1+|t|)^{1/3+\vepsilon}.
\end{equation*}
Recently, this was generalized  to modular forms $f (z)$ of arbitrary level and nebentypus  by the Jutila method (\cite{Jut87}) in \cite{BMN} and  by  the Bessel $\delta$-method in  \cite{AHLQ-Bessel}. 
Our approach follows the latter by a Bessel $\delta$-identity over $F$ established in \S \ref{sec: Bessel delta identity}, so the unramified assumption on $\uppii$ should be removable with some efforts. 

The reader is referred to \cite{FS-Maass,GLQi-Hybrid}  for the application of Bessel $\delta$-method in the hybrid subconvexity problem in both the $t$- and $q$-aspects. 

Finally, we remark that, similar to   \cite[Theorem 1.2]{Koyama-QUE} and \cite[Theorem 1.3]{Sarnak-P-QUE-H3}, in which they deduced from the $t$-aspect subconvexity the Quantum Unique Ergodicity (QUE) for (spherical) Eisenstein series $E (w; 1/2+it)$ on $\mathrm{PSL}_2 (\SO) \backslash \BH^3$, the hybrid subconvexity (as in \cite{Michel-Venkatesh-GL2,WuHan-GL2,Yang-Burgess-Bound} or Theorem \ref{1thm: Weyl bound}) would  imply QUE for (non-spherical) Eisenstein series on $\mathrm{PGL}_2(\SO) \backslash \mathrm{PGL}_2(\BC) $.  


\subsection*{Notation} 

By $X\Lt Y$ or $X=O(Y)$, we mean $|X|\leqslant cY$ for some constant $c>0$ while by $X\asymp Y$, we mean $X\Lt Y$ and $Y\Lt X$. We write $X \Lt_{\valpha, \beta, ...} Y $ or $  X = O_{\valpha, \beta, ...} (Y) $ if the implied constant $c$ depends on $\valpha$, $\beta$, ....  


For a large parameter $X$ (in most contexts $X = 1 + |t| + |m|$),  by `negligibly small' we mean $ O_A ( X^{-A} )$ for arbitrarily large but fixed $A > 0$. 

Throughout the paper,  $\vepsilon$ is arbitrarily small and its value  may differ from one occurrence to another.

	\section{Preliminaries}

	\subsection{Basic Notation}\label{sec: notation}
Let $F = \BQ ( \sqrt{d_{F}} )$ be an imaginary quadratic field with discriminant $d_{F} < 0$. We assume throughout that the class number $h_{F} = 1$. Let $\SO$ denote its ring of integers and $\SO^{\times}$ its group of units. 
Note that $ \SO' $ (the dual of $\SO$) is generated by $1/  \sqrt{d_F} $. 
Let $w_F = |\SO^{\times}|$. 
Let $\mathrm{N}$ and $\Tr$ denote the norm and the trace on $F \subset \BC$, respectively. 

Define $e[z]=e( \mathrm{Tr} (z))$ (as usual $e(x) = \exp (2\pi i x)$)  to be the standard additive character on $\BC$.  Let $\nd \mathrm{A} = i \nd z \! \wedge \! \nd \widebar{z}$ denote the area  form on $\BC$ (one may consider it as twice the ordinary Lebesgue measure). We shall usually write $\nd \mathrm{A} = \nd \mathrm{A} (z)$ so as to avoid ambiguity. 

It will be convenient to introduce $e_F [ z ] = e [ z / \sqrt{d_F}]$.

As the class number $h_F = 1$, we shall let $(n)$ denote non-zero integral ideals in $\SO$ for $n \in \SO \smallsetminus \{0\}$. 

For  $ (t, m) \in \BR \times w_F\cdot\BZ $, we define 
 the Gr\"ossencharakter $\upvchi_{it, m} $ on $ \SO \smallsetminus \{ 0 \} $: 
\begin{align} 
	\upvchi_{it, m} (n) = |n|^{it} (n / |n|)^{m}, 
\end{align} and the analytic conductor: 
\begin{equation}\label{2def: analytic conductor}
	\mathrm{C} (t,m)=1+|it+m|.
\end{equation}

\subsection{Definition of Bessel Functions for $\mathrm{PGL}_2(\BC)$} \label{2sec: Bessel}
For $(\vnu,\eqq)\in \BC\times\BZ$, let $\uppii_{\vnu,\eqq}$ be the associated principal series of $\mathrm{PGL}_2(\BC)$: the unique infinite dimensional constituent of the parabolic induction of the character
\begin{equation*}
	\upvchi_{\vnu,\eqq} : \begin{pmatrix}
		\sqrt{z} & v\\  & 1/\sqrt{z}
	\end{pmatrix}
	\ra  |z|^{\vnu} (z / |z|)^\eqq.
\end{equation*}
Note that $ \uppii_{\vnu,\eqq} \approx \uppii_{- \vnu, -\eqq}  $. It is well-known that the unitary representations of $\mathrm{PGL}_2(\BC)$ are classified into two categories:  for $ \vnu $ imaginary, $\uppii_{\vnu,\eqq}$ are the unitary principal series, while, for $ \vnu \in (0, 1/2) $,   $\uppii_{\vnu,0}$ are the unitary complementary series. 

As in \cite{Watson}, let $J_{\vnu} (z) $ be the Bessel function: 
\begin{align*}
	J_{\vnu} (z) = \sum_{n=0}^{\infty} \frac {(-1)^n (z/2)^{\vnu +2n}} {n! \Gamma (\vnu + n + 1)}. 
\end{align*}   
According to \cite[\S\S 15.3, 18.2]{Qi-Bessel} (the parametrization here has been simplified as the central character is trivial), let the Bessel function  attached to $\uppii_{\vnu,\eqq}$ be defined by
\begin{equation}\label{3eq: expression for Bessel kernel}
	\boldsymbol{J}_{\vnu,\eqq}(z) = \frac{2\pi^2} {\sin(\pi \vnu)}\left( J_{-\vnu,-\eqq}(4\pi z) - J_{\vnu,\eqq}(4\pi z) \right),
\end{equation}
where
\begin{equation}\label{3eq: defintion of complex J-Bessel}
	J_{\vnu,\eqq}(z) = J_{\vnu+\eqq}(z)J_{\vnu-\eqq}(\bar z).
\end{equation}
The Bessel functions $J_{\vnu,\eqq}(z) $ and  $\boldsymbol{J}_{\vnu,\eqq}(z)$ are even real analytic functions on $\BC \smallsetminus \{ 0 \}$. Alternatively, by the relations in \cite[3.61 (1, 2)]{Watson}, we may express $ \boldsymbol{J}_{\vnu,\eqq}(z) $ in terms of Hankel functions: 
\begin{equation}\label{2eq: expression for Bessel kernel via Hankel}
	\boldsymbol{J}_{\vnu,\eqq}(z)
	= \pi^2 i \big(e^{\pi i\vnu}H_{\vnu, \eqq}^{(1)}(4\pi z) 
	-e^{-\pi i\vnu}H_{\vnu, \eqq}^{(2)}(4\pi z) \big),
\end{equation}
where
\begin{equation}\label{3eq: Hankel}
	H^{(1, 2)}_{\vnu,\eqq}(z) = H^{(1, 2)}_{\vnu+\eqq}(z) H^{(1, 2)}_{\vnu-\eqq}(\bar z).
\end{equation}

\begin{remark}
	Actually,  in \cite[\S 18.2]{Qi-Bessel}, the attached Bessel function is normalized to be $ |z| \cdot \boldsymbol{J}_{\vnu,\eqq}(\sqrt{z}) $, as the integral kernel for the action of Weyl element in the Kirillov model of $\uppii_{\vnu,\eqq}$.  
\end{remark}

	\subsection{Cuspidal Representations on $  \mathrm{PGL}_2({\protect \SO}) \backslash\mathrm{PGL}_2(\BC) $}\label{2sec: Automorphic forms}

Fix an irreducible cuspidal automorphic representation $\uppii \subset L_{\text{cusp}}^2(\mathrm{PGL}_2(\SO)\backslash\mathrm{PGL}_2(\BC))$. Let $(\vnu,\eqq)$ be the Archimedean parameter of $\uppii$ so that   $\uppii\simeq\uppii_{\vnu,\eqq}$ as representations of $  \mathrm{PGL}_2(\BC) $. Assume that $ \uppii $ is Hecke invariant. Let $\lambda_\uppii(n)$ be the Hecke eigenvalues of $\uppii$. By the $\mathrm{PGL}_2(\SO)$-invariance, these   $\lambda_\uppii(n)$ are real-valued and depend only on the ideal $(n)$. Note that the Rankin--Selberg theory yields the Ramanujan bound on average:
\begin{equation}\label{4eq: Ramanujan bound}
	\sum_{\mathrm{N}(n) \leqslant X}\left|\lambda_\uppii(n)\right|^2\Lt_\uppii X.
\end{equation}
Moreover, it is well-known that $\uppii$ is self-dual since it has trivial central character. 

\subsection{Twisted $L$-functions} 

	For a Gr\"ossencharakter $\upvchi_{it, m} $, define the $L$-function associated with  $\uppii \times \upvchi_{it, m}$ by
\begin{equation*}
	L(s, \uppii \times \upvchi_{it, m}) = \sum_{(n) \subset \SO} \frac{\lambda_\uppii(n) \upvchi_{it, m}(n)}{\RN(n)^s}, \quad (\Re(s) > 1).
\end{equation*}
According to \cite{Knapp}, the gamma factor of $\uppii \times \upvchi_{it, m}$ is equal to $(2\pi)^{-2s}\gamma(s,\vnu, \eqq, t, m)$, with
\begin{equation*}
	\gamma(s, \vnu, \eqq,  t, m) = \Gamma\left(s +\frac{it+\vnu+|m+\eqq|}{2}\right)\Gamma\left(s+ \frac{it-\vnu+|m-\eqq|}{2}\right).
\end{equation*}
It is known that $L(s, \uppii \times \upvchi_{it, m})$ is entire and satisfies the functional equation
\begin{equation*}
	\Lambda(s, \uppii \times \upvchi_{it, m}) = \epsilon (q, m) \Lambda(1-s,  {\uppii} \times  {\upvchi}_{-it, -m}),
\end{equation*}
where the completed $L$-function 
\begin{equation*}
	\Lambda(s, \uppii \times \upvchi_{it, m}) = \sigma_F^{-2s} \gamma(s,\vnu, \eqq, t, m) L(s, \uppii \times \upvchi_{it, m}), \qquad \sigma_F = 2\pi / \sqrt{|d_F|},
\end{equation*}
and the root number $$\epsilon (\eqq, m) = (-1)^{\max \{|\eqq|, |m|\}}. $$

\subsection{Poisson Summation Formula} 

The following Poisson summation formula is a special case of \cite[VII.3 (3.2)]{Neukirch-ANT}.

\begin{lem}\label{lem: Poisson}
	Let $\varww \in C_c^{\infty} (\BC)$. Then 
	\begin{align}
		\frac 1 {\sqrt{|d_{F}|}}	\sum_{m \in \SO'} \varww (m) =\sum_{n \in \SO } \hat {\varww} (n),
	\end{align}
	where $ \hat {\varww} $ is the complex Fourier transform of $ \varww $ defined to be 
	\begin{align}\label{5eq: Fourier}
		\hat {\varww} (u) = \viint_{\BC} \varww (z) e[ {u} z] \nd \mathrm{A}, \qquad  (\text{$u \in \BC$}).
	\end{align}
\end{lem}

\begin{cor}\label{cor: Poisson}
	Let $\varww \in C_c^{\infty} (\BC)$. For non-zero $ c \in \SO$  let $\phi (\hskip 1pt \cdot \hskip 1pt | c)$ be a function on $\SO  / c \shskip \SO$. Then
	\begin{align}\label{5eq: Poisson Cor}
		\sum_{n \in \SO} \phi (n| c)  \varww (n) = \frac 1 {\sqrt{|d_{F}|} \RN (c)}  \sum_{ n \in \SO} \hat {\phi }  (n|c) \hat{\varww} (n/ \sqrt{d_F} c), 
	\end{align}
	where $\hat{\varww} $ is as defined in {\rm\eqref{5eq: Fourier}}, and
	\begin{align}
		\hat {\phi }  (n|c) =  \sum_{a \in \SO  / c \shskip \SO } \phi (a|c) e_F \Big[ \hskip -1pt   - \frac { {n} a } { c} \Big] .  
	\end{align}
\end{cor}

\begin{proof}
	We  split the left-hand side of \eqref{5eq: Poisson Cor} according to the residue classes in  $\SO  / c \shskip \SO$ as follows:
	\begin{align*}
		\sum_{a \in \SO / c \shskip \SO } \phi (a | c) \sum_{m \in \SO'} \varww (a  + \sqrt{d_F} c m  ) . 
	\end{align*}
	By applying the Poisson summation in Lemma \ref{lem: Poisson} to the inner sum, we obtain 
	\begin{equation*}
		\begin{split}
			& \sqrt{|d_{F}|} \sum_{a \in \SO / c \shskip \SO} \phi (a | c) \sum_{ n \in \SO} \viint_{\BC}  \varww ( a  +   \sqrt{d_F}  c z ) e [{n} z] \nd \mathrm{A} \\
		= {} & \sqrt{|d_{F}|} \sum_{a \in \SO  / c \shskip \SO} \phi (a | c) \sum_{ n \in \SO} \frac 1 {|\sqrt{d_F} c|^2 } \viint_{\BC}  \varww (  z) e \bigg[  \frac { {n} z } {\sqrt{d_F} c}  - \frac { {n} a } { \sqrt{d_F}  c}   \bigg]  \nd \mathrm{A},
		\end{split}
	\end{equation*}
	and hence the right-hand side of \eqref{5eq: Poisson Cor}. 
\end{proof}

\subsection{Vorono\"i Summation Formula}

We record here the Vorono\"i summation formula in \cite[Proposition 3.4]{Qi-Wilton}, which is a special case of \cite[Theorem 1]{Ichino-Templier} translated into the classical language.

	\begin{lem}\label{2lem: Voronoi summation formula}
	As in \S \ref{2sec: Automorphic forms}, let $\uppii$ be a cuspidal representation with Hecke eigenvalues $\lambda_\uppii(n)$ and spectral parameters $(\vnu,\eqq)$. Let $a,c\in\SO$ with $c\neq0$ and $(a,c)=1$. Let $\varww \in \allowbreak C_c^{\infty} (\BC \smallsetminus \{0\} )$.  Then 
	\begin{equation}
		\sum_{n\in\SO\smallsetminus\{0\}} \lambda_\uppii(n)e_F\lb \frac{n a}{c}\rb \varww(n)=\frac{1}{|d_F| \RN(c)} \sum_{n\in\SO\smallsetminus\{0\}} \lambda_\uppii(n) e_F\lb -\frac{n \overline{a}}{c}\rb \breve{\varww} \bigg(  \frac{n}{ {d_F}  c^2 }\bigg) ,
	\end{equation}
	where $\widebar{a} a\equiv 1 (\mod c)$, and $\breve{\varww}$ is the Hankel transform of $\varww$ defined to be 
	\begin{equation}
		\breve{\varww}(u) =  2\pi^2\viint_{\BC} \varww(z) \boldsymbol{J}_{\vnu,\eqq}(\sqrt{u z} )\nd \mathrm{A}, \qquad  (\text{$u \in \BC\smallsetminus \{0\}$}), 
	\end{equation}
	with $\boldsymbol{J}_{\vnu,\eqq}(z)$ the Bessel kernel defined in \S \ref{2sec: Bessel}.
\end{lem}

\begin{proof}
	It suffices to match our notation (in the case $h_F = 1$) with that in \cite[Proposition 3.4]{Qi-Wilton}. To this end,  set $\valpha = a$ and $\beta = c$; the condition $(a,c) = 1$ yields $\mathfrak{b} = (c)$. On the left side, we set $\gamma = n/\sqrt{d_F} \in \SO' $ so that $$A_\pi(\gamma\mathfrak{D}) = \lambda_\uppii(n), \quad f(\gamma) = \varww(n), \quad \psi_f \lp \frac {\alpha\gamma} \beta \rp = e_F\lb \frac{na} c \rb.$$ On the right side, we set $\gamma = n/ \sqrt{d_F} c^2 \in \mathfrak{b}^{-2} \SO'$ so that $$A_{\widetilde{\pi}}(\gamma\mathfrak{b}^2 \mathfrak{D}) = \lambda_\uppii(n), \quad \psi_{\mathfrak{b}}\left(-\frac {\widebar{\valpha}\beta^2\gamma} \beta\right) = e_F\left[ - \frac {n \widebar{a}} c \right], \quad \widetilde{f}(\gamma) = \frac 1 {|d_F| } \breve{\varww} \bigg(\frac n  {d_F c^2}  \bigg). $$
	Note here that $\uppii$ is self-dual as it has trivial central character. 
\end{proof}

\section{Reduction via Approximate Functional Equation}

 The  approximate functional equation for $L(1/2,\uppii\times \upvchi_{it,m}) $ in the next lemma is a direct consequence of \cite[Theorem 2.1]{Harcos-AFE}.

\begin{lem}\label{3lem: Approximate Functional Equation}
  Let $ F   \in C^\infty(\BR_+)$ be real-valued, such that $ F(x)+ F(1/x)=1$ and that $ F^{(j)} (x) $ decays  rapidly as $ x \ra  \infty$ for every $j$. 
	Then
	\begin{equation*}
		\begin{split}
			L(1/2,\uppii\times \upvchi_{it,m})   = & \sum_{(n)\subset\SO}\frac{\lambda_\uppii(n)\upvchi_{it,m}(n)}{|n|} F \bigg(\frac{|n|}{\RC(t,m)} \bigg)\\
			+  \epsilon (\vnu, t ; q, m)  & \sum_{ (n)\subset\SO}\frac{\lambda_\uppii(n)\upvchi_{-it,-m}(n)}{|n|} F \bigg(\frac{|n|}{\RC(t,m)}\bigg)+O_{\vepsilon, \uppii , F }\big(   {\RC(t,m)^{  \vepsilon}} \big),
		\end{split}
	\end{equation*}
	where $\epsilon (\vnu, t ; q, m) $ is of unity norm, explicitly given by 
	\begin{equation*}
		\epsilon (\vnu, t ; q, m) = \epsilon ( q, m) \cdot   \frac{\gamma(1/2, \vnu, \eqq, - t, - m)}{ \gamma(1/2 , \vnu, \eqq,  t, m) } . 
	\end{equation*}
\end{lem}

By applying a dyadic partition of unity to the approximate functional equation above, we infer that 
\begin{align*}
		L(1/2,\uppii\times \upvchi_{it,m})
	\Lt 
	\RC(t,m)^\vepsilon
	\bigg( 
	\frac{|S_{\uppii} (N; t, m)|}{N}
	+ 1
	\bigg) , 
\end{align*}
for some $ N < \RC(t,m)^{1+\vepsilon} $, where 
\begin{align}\label{3eq: reduction after AFE}
	S_{\uppii} (N; t, m) =  \sum_{n \in \SO} \lambda_\uppii(n)\upvchi_{it,m}(n)V \bigg(\frac {|n|} {N}  \bigg) , 
\end{align}
and   $ V \in C_c^{\infty} [1, 2]$ is some weight function with bounds $ V^{(j)} (x) \Lt_{j}  1$. Note that $ S_{\uppii} (N; t, m) \allowbreak = O_{\uppii} (N^2 ) $ follows trivially from the Cauchy inequality and the averaged Ramanujan bound in \eqref{4eq: Ramanujan bound}. 

The hybrid Weyl subconvexity bound in Theorem \ref{1thm: Weyl bound} is a direct consequence of  the bound below for $S_{\uppii} (N; t, m) $ in Theorem \ref{thm: bound for S(N)} as it implies 
\begin{align*}
	 \frac{|S_{\uppii} (N; t, m)|}{N} \Lt \RC(t,m)^{2/3+\vepsilon}, 
\end{align*}
for any  $\RC(t,m)^{2/3+\vepsilon} < N < \RC(t,m)^{1+\vepsilon} $ (the trivial bound is better and sufficient for   $ N \leqslant \RC(t,m)^{2/3+\vepsilon}$).

\begin{theorem}\label{thm: bound for S(N)}
	 Let $ V \in C_c^{\infty} [1, 2]$ such that $ V^{(j) } (x) \Lt_j 1 $. Then for $ N < \RC(t,m)^{1+\vepsilon} $, we have 
	 \begin{align}\label{3eq: bound for S(N)}
	 	\sum_{n \in \SO} \lambda_\uppii(n)\upvchi_{it,m}(n)V \bigg(\frac {|n|} {N}  \bigg) \Lt_{\vepsilon, \uppii} N^{1+\vepsilon} \RC(t,m)^{2/3} . 
	 \end{align}
\end{theorem}

\begin{remark}
	Similar to \cite[Theorem 1.1]{AHLQ-Bessel}, for all values of $N$  we may prove with some effort that
	\begin{align}
		\sum_{n \in \SO} \lambda_\uppii(n)\upvchi_{it,m}(n)V \bigg(\frac {|n|} {N}  \bigg) \Lt_{\vepsilon, \uppii}   N^{1+\vepsilon} \RC(t,m)^{2/3} + 
		\frac{N^{2+\vepsilon}}{\RC(t,m)^{1/3}}.
	\end{align}
\end{remark}

Henceforth, for simplicity,  let us write $S (N) = S_{\uppii} (N; t, m) $ and introduce 
\begin{equation}\label{3eq: rho, theta}
	\rho(z) = \frac{\log|z|}{2\pi},\qquad \theta(z) = \frac{\arg(z)}{2\pi}, 
\end{equation}
so that $S (N) $ turns into the exponential sum
\begin{equation}\label{3eq: S(N)}
	S (N) = \sum_{n \in \SO} \lambda_\uppii(n) e (  t \rho (n) + m \theta (n) ) V \bigg(\frac {|n|} {N}  \bigg) . 
\end{equation}

\section{Analysis of Integrals by Stationary Phase and Weil Identity}\label{sec: analysis of integrals}

In this section, we shall apply the method of stationary phase and the identity of Weil to analyze certain oscillatory integrals that will arise after the applications of Poisson summation. The material here is technical and can be safely skipped on the first reading. 

The lemmas that follow will be stated in terms of Wirtinger derivatives:
\begin{equation*}
	\frac{\partial}{\partial z}=\frac1 2 \bigg( \frac{\partial}{\partial x}-i\frac{\partial}{\partial y} \bigg),\quad  \frac{\partial}{\partial \bar z}=\frac1 2 \bigg( \frac{\partial}{\partial x}+i\frac{\partial}{\partial y}\bigg) , \qquad z = x+iy.
\end{equation*} 
Note that 
\begin{align*}
	\frac{\partial^2 f}{\partial z^2} - \frac{\partial^2 f}{\partial \bar{z}^2} = - i	\frac {\partial^2 f} {\partial x \partial y} , 
\end{align*}
\begin{align*}
\Delta f =	4 \frac {\partial^2 f} {\partial z \partial \bar{z}} =  \frac{\partial^2 f}{\partial x^2} + \frac{\partial^2 f}{\partial y^2} ,     
\end{align*}
\begin{align*}
	  \det \nabla^2 f = 4 \bigg(  \bigg(\frac{\partial^2 f}{\partial z \partial \bar z} \bigg)^2 - \frac{\partial^2 f}{\partial z^2} \cdot \frac{\partial^2 f}{\partial \bar{z}^2}  \bigg) = \frac{\partial^2 f}{\partial x^2} \cdot \frac{\partial^2 f}{\partial y^2} - \bigg(\frac{\partial^2 f}{\partial x \partial y} \bigg)^2,
\end{align*}
where $ \Delta f  $ is the Laplacian and $ \nabla^2 f $ is the Hessian matrix of $f$.

\subsection{Stationary Phase Lemmas}

First,   \cite[Lemma 7.4]{Qi-GL(3)} is rephrased as follows in the Wirtinger notation.   

	\begin{lem}\label{4lem: Stationary Phase}
	Let $D \subset \BC$ be a bounded domain. Let $\varww \in C_c^{\infty} (D)$. Let $f \in C^{\infty} (D) $  be real-valued. Suppose that there are parameters $P, Q, \varUpsilon, \varPhi, R, S, Z > 0$ such that
	\begin{align*} 
		\frac{\partial^{i+j} f(z)}{\partial z^{i} \partial \bar{z}^{j}}  \Lt_{ i,j } \frac{Z}{Q^{i} \varPhi^j}, \qquad \frac{\partial^{k+l} \varww (z)}{\partial z^{k} \partial \bar{z}^{l}} \Lt_{k,l}  \frac{S}{P^{k} \varUpsilon^l},
	\end{align*}
	for all integers $i,j,k,l \geqslant 0$ with $i+j \geqslant 2$, and 
	\begin{align*}
		\left| \frac{\partial f (z)}{\partial z} \right|  \Gt  R . 
	\end{align*}
	Then   
	\begin{align*}
		\viint_D \! e ( f(z) )  \varww ( z ) \nd \mathrm{A} \Lt_{A}   \textit{Area}(D) S \bigg\{ \frac{1}{R} \bigg(  \frac{1}{P}+\frac{1}{\varUpsilon}+\frac{1}{Q}+\frac{1}{\varPhi} \bigg) \! + \frac{Z^2}{R^3}\bigg(  \frac{1}{Q^3}+\frac{1}{\varPhi^3}\bigg) \! \bigg\}^A
	\end{align*}
	for any $A \geqslant 0$.
\end{lem}

For our applications, we require smooth variants of the multi-dimensional second derivative test due to Srinivasan \cite[Lemma 4]{Srinivasan-Lattice-2}. For this, it requires some finiteness hypotheses on the integral domain and the phase function. 

\begin{hyp}\label{4hyp: domain}
	Assume that $D$ is a bounded region in a Euclidean space, enclosed by finite algebraic surfaces of bounded degree. 
\end{hyp}

\begin{hyp}\label{4hyp: phase} 
	Assume that $f$ is a real-valued smooth function whose first and second derivatives are algebraic of bounded degree. 
\end{hyp}


\begin{lem}\label{4lem: SDT-2 dimension}
	Let $D \subset \BC$  and $f\in C^{\infty} (D) $ satisfy Hypotheses \ref{4hyp: domain} and \ref{4hyp: phase}.  Suppose that $\Delta f = 0$ and that 
	\begin{equation*}
		\left| {\partial^2 f} / {\partial z^2}\right|\Gt \lambda > 0.
	\end{equation*} 
Let $\varww \in C_c^{\infty} (D)$. Define its variation 
	$$ V  = \viint_D \big(\big|  {\partial^2 \varww(z)}/ {\partial z^2} \big| + \big|  {\partial^2 \varww(z)} / {\partial \bar{z}^2} \big| \big) \nd \mathrm{A}. $$ 
	Then 
	\begin{equation*}
		\viint_D e (f(z) )\varww(z) \nd \mathrm{A} \Lt \frac{V}{\lambda}.
	\end{equation*}
\end{lem}

\begin{proof}
For  $f$ harmonic and real-valued, $\partial^2 f /\partial x^2 = - \partial^2 f /\partial y^2 = 2 \mathrm{Re}   \big(\partial^2 f /\partial z^2 \big)$ and the  Hessian determinant $\det \nabla^2 f =-4 \big|\partial^2f/\partial z^2\big|^2$.  Hence $\big|\partial^2 f /\partial x^2\big| =\big|   \partial^2 f /\partial y^2\big| \Gt \lambda$ and $  |\det \nabla^2 f  | \Gt \lambda^2 $, and it follows from \cite[Lemma 4]{Srinivasan-Lattice-2} that
	 \begin{align*}
	 	 \viint_D e (f(z) ) \nd \mathrm{A} \Lt \frac{1}{\lambda}. 
	 \end{align*}
 The key point is that the implied constant is independent of $D$ so that the bound applies uniformly to the sub-domains
 \begin{align*}
 	D(x,y)=\{z \in D: \mathrm{Re}(z) \leqslant x,\,\mathrm{Im}(z) \leqslant y\}. 
 \end{align*}
Note that Hypothesis \ref{4hyp: domain} still holds for $D (x, y)$. Therefore 
\begin{align*}
	 \viint_{D(x,y)} e(f(z)) \nd \mathrm{A} \Lt \frac 1 {\lambda} . 
\end{align*}
Let $E(x, y) $ denote the integral on the left. By partial integration twice,  
\begin{align*}
	\viint \! e\big(f(z)\big)\varww(z) \nd \mathrm{A} & =  2 \! \viint \! E(x, y) \frac{\partial^2 \varww (x, y)}{\partial x \partial y} \nd x \nd y   \Lt \frac{1}{\lambda} \! \viint  \bigg|\frac{\partial^2 \varww(z)}{\partial z^2} - \frac{\partial^2 \varww(z)}{\partial \bar{z}^2} \bigg|\nd \mathrm{A} \Lt  \frac {V} {\lambda} . 
\end{align*}
\end{proof}

\begin{lem}\label{4lem: SDT-4 dimension} 
Let $D \subset \BC $  and $f \in C^{\infty} (D^2) $ satisfy Hypotheses \ref{4hyp: domain} and \ref{4hyp: phase}. 	Assume that $f$ is of the form
	\begin{equation*}
		f(z_1, z_2) = f_1(z_1) + f_2(z_2) + \delta(z_1, z_2),
	\end{equation*}
	for $f_1 (z), f_2 (z)$, and $\delta (z_1, z_2)$ harmonic and real-valued. Assume further that there is $ \lambda > 1$ such that 
	\begin{equation*}
		\left| \frac {\partial^2 f_1 (z)}   {\partial z^2}\right|,  \,   \left| \frac {\partial^2 f_2 (z)} {\partial z^2}\right|  \Gt \lambda , \qquad 
	\frac{\partial^{j_1+k_1+j_2+k_2}\delta(z_1,z_2)}
	{\partial z_1^{j_1}\partial\bar z_1^{k_1}
		\partial z_2^{j_2}\partial\bar z_2^{k_2}}
	\Lt 
	\lambda^{1-\vepsilon}, 
\end{equation*}
for $j_1+k_1+j_2+k_2 =2$. Let $\varww_1, \varww_2 \in C_c^{\infty} (D)$. Define their variations $V_1$ and $V_2 $ as in Lemma \ref{4lem: SDT-2 dimension}. Then 
\begin{equation*}
	\viint_{D} \! \viint_D
	e (f(z_1,z_2) )\varww_1(z_1) \varww_2 (z_2)
	\nd\mathrm{A}(z_1)\,\nd\mathrm{A}(z_2)
	\Lt\frac{V_1 V_2} {\lambda^2}.
\end{equation*}
	\end{lem}
 
	 Let us omit the proof as it is similar to that of Lemma \ref{4lem: SDT-2 dimension}. Given the form of $f (z_1, z_2)$ as above, the principal minor determinants of the Hessian matrix are easy to estimate. Note that $f  $  is referred to as an `almost separable' phase in \cite{HMQ-Beyond-Weyl}.   


	\subsection{Analysis of Oscillatory Integrals I}


First, we   consider 
\begin{equation*}
	\text{\large $\EuScript{I}$}(\valpha)
	=
	\viint V(|z|)e(f(z;\valpha))\nd\mathrm{A} ,
\end{equation*}
where  $V\in C_c^\infty(\BR_+)$ is fixed, and  $f(z;\valpha)$ is of the form
\begin{equation*}
	f(z; \valpha)
	=
	2 t\rho(z)+2m\theta(z)+\Tr (\valpha z^2 +\delta(z) ), \qquad \rho(z) = \frac{\log|z|}{2\pi},\quad \theta(z) = \frac{\arg(z)}{2\pi} .
\end{equation*}

\begin{lem}\label{4lem: pre-analysis of I}
	On  the annulus  support of $V (|z|)$, assume that $\delta  $ is holomorphic and
	\begin{equation*}
		{\partial^j\delta} / {\partial z^j}
		\Lt_j\RC(t,m)^{1-\vepsilon}. 
	\end{equation*}
	Then	the integral  $\text{\large $\EuScript{I}$}(\valpha)$ is negligibly small unless
	$|\valpha|\Lt\RC(t,m)$, in which case  \begin{align*}
		\text{\large $\EuScript{I}$}(\valpha)\Lt \frac 1{ \RC(t,m)} . 
	\end{align*} 
\end{lem}

\begin{proof}
	By a suitable smooth radial partition of unity, we may restrict the integration to a sector on which a branch of $ z \ra \sqrt{z} $ is a biholomorphism. After the change of variable  $ z \ra \sqrt{z} $, on the sector of support, the weight function is still of bounded derivatives, while  the phase function reads 
	\begin{align*}
		g (z; \valpha) = t\rho(z)+ m \theta(z)+\Tr (\valpha z  +\delta(\sqrt{z}) ),
	\end{align*}
It is clear that $g (z; \valpha)$ is real-valued, harmonic, and 
	\begin{equation*}
		\frac {\partial g(z;\valpha)} {\partial z} =\frac{t-im}{4\pi z}+\valpha+O \big(\RC(t,m)^{1-\vepsilon} \big) , \quad 
	\frac {\partial^2 g(z;\valpha)} {\partial z^2} =-\frac{t-im}{4\pi z^2}+O \big(  \RC(t,m)^{1-\vepsilon} \big).
	\end{equation*}
For $|\valpha|\Gt\RC(t,m)$, we have $ |\partial g(z;\valpha)/\partial z |\Gt |\valpha|$, and hence $\text{\large $\EuScript{I}$}(\valpha)$ is negligibly small by Lemma \ref{4lem: Stationary Phase}. Since $  \big|\partial^2 g(z;\valpha)/\partial z^2 \big|\Gt \RC(t,m) $, the bound $ \text{\large $\EuScript{I}$}(\valpha) = O    (1 / { \RC(t,m)} ) $ is a direct consequence of Lemma \ref{4lem: SDT-2 dimension}. 
\end{proof}

Next,  we consider 
	\begin{equation*}
	\text{\large $\EuScript{I}$} (\valpha_1,\valpha_2) = \viint \! \viint V_1(|z_1|)V_2(|z_2|) 
	e(f (z_1,z_2;\valpha_1,\valpha_2))\nd \mathrm{A}(z_1) \, \nd \mathrm{A}(z_2),
\end{equation*}
where  $V_1, V_2 \in C_c^\infty(\BR_+)$  and the phase function is of the form
\begin{equation*}
	f (z_1,z_2;\valpha_1,\valpha_2) = 2 t(\rho(z_1)-   \rho(z_2))+ 2m(\theta(z_1)-\theta(z_2))+ \Tr\big(\valpha_1z_1^2-\valpha_2z_2^2 \big) + \Tr\, \delta(z_1,z_2) . 
\end{equation*}

\begin{lem}
	\label{4lem: analysis of I, 2} 
	On  the annulus  support of $V_1 (|z_1|) V_2 (|z_2|)$, assume that 
	$\delta $ is holomorphic and that
	\begin{equation*}
		\frac{\partial^{j_1+j_2}\delta(z_1,z_2)}
		{\partial z_1^{j_1}	\partial z_2^{j_2}}
		\Lt_{j_1,j_2} 
		\RC(t,m)^{1-\vepsilon} .
	\end{equation*} Then $\text{\large $\EuScript{I}$} (\valpha_1,\valpha_2)$ is negligibly small unless $|\valpha_1|,|\valpha_2|\Lt\RC(t,m)$, in which case $$\text{\large $\EuScript{I}$} (\valpha_1,\valpha_2)\Lt \frac 1{ \RC(t,m)^2} . $$
\end{lem}

\begin{proof}
	The proof is parallel to that of Lemma \ref{4lem: pre-analysis of I}: one just needs to apply Lemma \ref{4lem: SDT-4 dimension}  in place of Lemma \ref{4lem: SDT-2 dimension}. 
\end{proof}

	\subsection{The Weil Identity} 
	
	In the case that the phase function is quadratic, we may use the Weil--Plancherel identity as a substitute for stationary phase
	over $\BC$: 
	\begin{equation}\label{4eq: Weil}
		\viint_{\BC}\varww(z)e\left[\frac{uz^2}{2}\right]\nd\mathrm{A} 
		=
		\frac{1}{|u|}
		\viint_{\BC}\hat{\varww}(z)
		e\left[-\frac{z^2}{2u}\right]\nd\mathrm{A} ,
		\qquad u\in\BC^\times ; 
	\end{equation} see \cite{Weil-Unitary} and
	\cite[(26)]{Jacquet-RTF}. The advantage of this identity is that it would enable us to derive easily estimates for the derivatives of the integral.

	\begin{lem}\label{4lem: derivative bound}
	Let $H \subset \BC$ denote the complement of a fixed disc centered at the origin and  $D\subset\BC$ be a fixed bounded domain. Let  
		$\varww(z; v )\in C^\infty(D \times H )$. Assume that   $\varww(z; v )$ is supported in a fixed compact subset of $D$ for all $v \in H$ and that
		\begin{equation*}
			 v ^{i+j}   \frac{\partial^{i+j+k+ l  }\varww(z; v )}
			{\partial z^k\partial\bar z^l  \partial v ^i\partial\bar v ^ j   }
			\Lt_{i,j,k, l  } S, 
		\end{equation*}
		for any $i,j,k, l  \geqslant0$. Define
		\begin{equation*}
			I( v )
			=
			\viint_D\varww(z; v )e [ v  z^2  ]
			\nd\mathrm{A}.
		\end{equation*}
		Then, for every $i, j  \geqslant0$,
		\begin{equation*}
		v^{i+j}	\frac{\partial^{i+ j }I( v )}
			{\partial v ^i \partial\bar v ^ j  }
			\Lt_{i, j }
			\frac{S}{| v | }.
		\end{equation*}
	\end{lem}

\begin{proof}
	It follows from \eqref{4eq: Weil} that 
	\begin{equation*}
		I( v )
		=
		\frac{1}{ |2  v |}
		\viint_{\BC}
		\hat{\varww}(z; v )
		e\left[-\frac{z^2}{4 v }\right]
		\nd\mathrm{A} . 
	\end{equation*}  
For every $A \geqslant 0$, we have 
\begin{equation*}
	v ^{i+j}  \frac{\partial^{i+ j }\hat{\varww}(z; v )}
	{\partial v ^i\partial\bar v ^ j }
	\Lt_{i, j ,A} S(1+|z|)^{-A},
\end{equation*}
while, for any $|v| \Gt 1$, we have 
\begin{equation*}
	v ^{i+j} 
	\frac{\partial^{i+ j } e [- {z^2} / {4 v } ]}
	{\partial v ^i\partial\bar v ^ j } 	\Lt_{i, j }(1+|z|)^{2i+2 j }.
\end{equation*}
Thus the proof is completed if we differentiate under the integral and apply the bounds above after the product rule.   
\end{proof}

	\subsection{Analysis of Oscillatory Integrals II} 
	
	\begin{lem}
		\label{4lem: Fourier integral}
	Let $H \subset \BC$ be as in Lemma \ref{4lem: derivative bound}. 	Let $V\in C_c^\infty(\BR_+)$ be fixed, and define
		\begin{equation*}
			\text{\large $\EuScript{K}$}(v, w)
			= \viint_{\BC}V(|z|)e[vz^2 - 2 w z]\nd\mathrm{A} . 
		\end{equation*} 
	
	{\rm(1)} For $v \in H$, we have
	$ \text{\large $\EuScript{K}$}(v,w) \Lt_{A} (1+|v|+|w|)^{-A}$ for any $A \geqslant 0$ unless $|v|\asymp  |w|$, in which case, if we set
	\begin{equation*}
		u=\frac{w^2}{ v},
		\qquad
		\text{\large $\EuScript{W}$}(u; v )=e[u] \, \text{\large $\EuScript{K}$}(v, w), 
	\end{equation*}
	then, for every $i, j  \geqslant0$, 
	\begin{equation*}
		u^{i+j} 
		\frac{\partial^{i+ j} \text{\large $\EuScript{W}$}(u; v )}
		{\partial u^i\partial\bar u^ j }
		\Lt_{i, j }\frac{1}{|v|}.
	\end{equation*}

{\rm(2)} For  $v=0$, we have $\text{\large $\EuScript{K}$}(0, w)= \text{\large $\EuScript{W}$}_0(w)$ for some radial Schwartz function $\text{\large $\EuScript{W}$}_0$. 
	\end{lem}

\begin{proof}
	Part (1)  follows easily from Lemmas  \ref{4lem: Stationary Phase} and \ref{4lem: derivative bound}, with the aid of the change $ z \ra w / v \cdot (z + 1) $. Part (2) is obvious as $\text{\large $\EuScript{K}$}(0, w)$ is a Fourier integral. 
\end{proof}

\section{\texorpdfstring{A Bessel $\delta$-identity}{A Bessel \unichar{"03B4}-identity}}\label{sec: Bessel delta identity}

In this section, in parallel to  \cite[\S 3]{AHLQ-Bessel}, we establish the Bessel $\delta$-identity over imaginary quadratic fields. 

\subsection{Bessel Functions for $\mathrm{PGL}_2(\BC)$}
Recall the expressions of $\boldsymbol{J}_{\vnu,\eqq}(z)$  as in \eqref{3eq: expression for Bessel kernel}--\eqref{3eq: Hankel}.  Firstly, by the asymptotic expansions of Hankel in \cite[7.2 (1, 2)]{Watson}, we have
\begin{equation}\label{6eq: asymptotic behavior for Bessel kernel}
	\boldsymbol{J}_{\vnu,\eqq}(z) = \frac{W_{\vnu,\eqq}(z)}{|z|} e[2z]  + \frac{W_{\vnu,\eqq}( - z)}{|z|} e[-2z] + O_A\left(|z|^{-A}\right),
\end{equation}
for any fixed $A>0$, where  $W_{\vnu,\eqq}(z)$ (it depends implicitly on $A$) has bounds 
\begin{equation*}
	z^j\bar z^k\frac{\partial^{j+k} W_{\vnu,\eqq}(z) }{\partial z^j\partial\bar z^{k} } \Lt_{j,k,\vnu, q, A} 1. 
\end{equation*}
Next, we record here a Fourier--Mellin type integral formula for Bessel kernels (see \cite[Corollary 1.3]{Qi-BE}). 
	\begin{lem}\label{6lem: formula for Bessel kernels}
	For $|\Re(\vnu)|<\Re(\rho)<1/2$, we have
	\begin{equation}\label{3eq: formula for Bessel kernels}
		\begin{split}
			& \viint   \boldsymbol{J}_{\vnu,\eqq}(z)  e[-2z] |z|^{2\rho-2}  { \nd \mathrm{A} }  =  ({4}/{\pi }  )  (8\pi)^{-2 \rho}\cos(\pi \rho)\Gamma(1/2- \rho)^2 \\
			& \cdot \sin(\pi(\rho+\vnu))\sin(\pi(\rho-\vnu))   \Gamma(\rho+\vnu+\eqq)\Gamma(\rho+\vnu-\eqq)\Gamma(\rho-\vnu+\eqq)\Gamma(\rho-\vnu-\eqq).
		\end{split}
	\end{equation}
\end{lem}

	\subsection{Bessel $\delta$-identity}
	
	Let $a,b \neq 0$ be  complex numbers and $X>1$ be a large parameter. Fix a smooth real-valued  bump function $U \in C_c^{\infty} (\BR_+)$ and consider the following Bessel integral: 
	\begin{equation}\label{3eq: definition of I-integral} 
			I_{\vnu,\eqq}(a,b;X)  =  \viint U ( {|z|} / {X} ) \big(e[ 2a\sqrt z] + e[- 2a\sqrt z] \big) \boldsymbol{J}_{\vnu,\eqq}(b\sqrt z)   \nd \mathrm{A} .   
	\end{equation}
Recall here that $ \boldsymbol{J}_{\vnu,\eqq}(z) $ is an even function.  By the change $ \pm \sqrt{z} \ra z $, we have
\begin{equation}\label{3eq: definition of I-integral, 2} 
	I_{\vnu,\eqq}(a,b;X) = 4\viint  U ( {|z|^2} / {X} ) e[2az]\boldsymbol{J}_{\vnu,\eqq}(b z) |z|^2   \nd \mathrm{A} .
\end{equation}

Hereafter, the implied constants will always depend on $\vnu$, $\eqq$, and $U$, and, for simplicity, these will be suppressed from the subscripts of $O,\Lt$ or $\Gt$. 

We first consider the case $b=\pm a$. After the change of variables $z \rightarrow - z/a$, it follows from Mellin inversion that
\begin{equation}\label{6eq: Bessel integral after mellin inversion}
	I_{\vnu,\eqq}(a,\pm a;X) = \frac{2}{\pi i|a|^4}\int_{(\sigma)}  \widetilde{U}(s) |a^2 X|^s \nd s \viint \boldsymbol{J}_{\vnu,\eqq}(z)e[- 2z]  |z|^{2 -2s} { \nd \mathrm{A}} , 
\end{equation}
for $3/2 < \sigma < 2 - |\mathrm{Re}(\vnu)|$, 
where as usual  $(\sigma)$ denotes the vertical line $\Re(s)=\sigma$, and  $\widetilde{U}(s)$ is the   Mellin transform of $U (x)$ on $\BR_+$: 
\begin{equation*}
	\widetilde{U}(s) = \int_{\BR_+} U(x)x^{s-1}\nd x . 
\end{equation*} 

\begin{lem}\label{6lem: I (a,a)}
	Assume $ |a^2 X | > 1$.  Write 
	\begin{equation}\label{6eq: reformulate of Bessel integral}
		I_{\vnu,\eqq}(a,\pm a;X)=2\pi X^2I_{\vnu,\eqq} ( |a^2X | ).
	\end{equation}
	Then $I_{\vnu,\eqq}(x)$ is a smooth function on $\BR_+$ and satisfies
	\begin{equation}\label{6eq: upper bound for 1/I(x)}
		x^j\frac{\nd^j}{\nd x^j}\frac{1}{I_{\vnu,\eqq}(x)}\Lt_j \sqrt x, \qquad x\Gt1.
	\end{equation} 
\end{lem}

\begin{proof}
	By definition,
	\begin{equation*}
		I_{\vnu,\eqq}(x)= \frac{1}{\pi^2 i} \int_{(\sigma)} \widetilde{U}(s) x^{s-2} \nd s \viint \boldsymbol{J}_{\vnu,\eqq}(z)e[- 2z] |z|^{2-2s} \nd \mathrm{A} ,
	\end{equation*}
	 Now we apply Lemma \ref{6lem: formula for Bessel kernels} with $\rho = 2 - s$ to evaluate the inner integral and observe that the right side of \eqref{3eq: formula for Bessel kernels} has simple poles at $s = 1/2,\, 3/2, ...$ from $ \cos(\pi (2-s) )\Gamma(s-3/2)^2$. 
	By shifting the contour to $\Re(s)=0$, collecting the residues at $s=1/2, 3/2$, and estimating the resulting integral, we obtain the following asymptotic formula: 
	\begin{equation*}
		I_{\vnu,\eqq}(x)= \frac{\widetilde{U}(3/2)}{\sqrt x} + O\lp\frac{1}{x\sqrt x}\rp.
	\end{equation*}
	It follows that for $x \Gt 1$ large,
	\begin{equation*}
		|I_{\vnu,\eqq}(x)|\Gt 1/\sqrt x.
	\end{equation*}
	Similarly, 
	\begin{equation*}
		x^jI_{\vnu,\eqq}^{(j)}(x)=\frac{\widetilde{U}(3/2)(2j-1)!!}{(-2)^j \sqrt x} + O_j\lp\frac{1}{x\sqrt x}\rp, 
	\end{equation*}
	so the derivatives of $I_{\vnu,\eqq} (x)$ have upper bounds: 
	\begin{equation*}
		x^jI_{\vnu,\eqq}^{(j)}(x)\Lt_j  1/\sqrt x. 
	\end{equation*}
Finally, the bounds for the derivatives of $1/ I_{\vnu,\eqq} (x)$ follow directly from the lower and upper bounds above for $I_{\vnu,\eqq}(x)$.  
\end{proof}

As for the off-diagonal case $b\neq\pm a$, we assume $|b^2 X|>1$ so that $ \boldsymbol{J}_{\vnu,\eqq}(b z) $ is oscillatory for $|z|^2 \asymp X$. By inserting \eqref{6eq: asymptotic behavior for Bessel kernel} into \eqref{3eq: definition of I-integral, 2}, up to a negligible error, we obtain two Fourier-type integrals  and hence the next lemma.  

	\begin{lem}\label{6lem: Bessel kernel is negligibly small}
	Assume $|b^2 X|>1$. Then the Bessel integral $I_{\vnu,\eqq}(a,b;X)$ is negligibly small if $|a+b|\sqrt{X} > X^\vepsilon$ and $|a-b|\sqrt{X} > X^\vepsilon$. 
\end{lem}

	Now we are ready to state the main result of this section. 
	
		\begin{theorem}\label{6thm: Bessel identity}
		Let $p \in \SO \smallsetminus \{0\}$ and $N,X>1$ be large parameters such that $X\Gt |p|^2/N$ and $X^{1-\vepsilon}>N$.  %
		Then for  $r,n\in\SO $ with $|r|,|n|\asymp N$, we have  
		\begin{equation}\label{6eq: delta identity}
			\begin{split}
				\delta(r = n)  = \frac{\mathrm{N}(r)^{1/4}}{ \mathrm{N}(p)^{3/2}X^{3/2}} \! \sum_{a\in\SO/p\SO} \! e_F\bigg[\frac{a(r-n)}{p}\bigg] I_{\vnu,\eqq}\bigg(\frac{\sqrt r}{p},\frac{\sqrt n}{p};X\bigg)V_{\vnu,\eqq}\bigg( \bigg| \frac{r X }{p^2} \bigg|\bigg) & \\
				 + O_{ A} \big( X^{-A}\big) &,
			\end{split}
		\end{equation}
	for any $A > 0$,	where   $\delta(r = n)$ is the Kronecker $\delta$-symbol that detects $r=n$, and
		\begin{equation*}
			V_{\vnu,\eqq}(x)= \frac{1}{2 \pi  \sqrt x I_{\vnu,\eqq}(x)}
		\end{equation*}
		has bounds $x^{j}V_{\vnu,\eqq}^{(j)}(x)\Lt_j 1$ for any $x \Gt 1$. 
	\end{theorem}

\begin{proof}
	By orthogonality,  
	\begin{equation*}
		\frac{1}{\mathrm{N}(p)}\sum_{a\in\SO/p\SO}e_F\left[\frac{a(r-n)}{p}\right] = \delta(r-n \in p\SO ).
	\end{equation*}
	According to Lemma \ref{6lem: Bessel kernel is negligibly small}, the integral $I_{\vnu,\eqq}(\sqrt r/p,\sqrt n/p;X)$ is negligibly small unless $|\sqrt r+\sqrt n| \Lt |p|X^\vepsilon/\sqrt X$ or $|\sqrt r-\sqrt n| \Lt |p|X^\vepsilon/\sqrt X$. Since $|r|,|n| \asymp N$, 
	in either case, 
	\begin{equation*}
		|r-n| \Lt |p|X^\vepsilon\sqrt{N/X}.
	\end{equation*}
	For $X^{1-\vepsilon}>N$, this  implies (for $X$ large) $$| r - n | < \min \big\{ |\omega | :  \omega \in p\SO \smallsetminus \{0\} \big\} ,$$ and hence forces $r=n$ in the case that $r-n \in p\SO $. Finally, in view of \eqref{6eq: reformulate of Bessel integral} in Lemma \ref{6lem: I (a,a)},  the identity in \eqref{6eq: delta identity} follows by suitable normalization.  
\end{proof}

	\section{Application of Bessel \texorpdfstring{$\delta$}{\unichar{"03B4}}-identity and Vorono\"i Summation}
	
	Consider the character sum in \eqref{3eq: reduction after AFE}:
	\begin{equation}
		S  (N ) =  \sum_{n \in \SO} \lambda_\uppii(n)\upvchi_{it,m}(n)V \bigg(\frac {|n|} {N}  \bigg). 
	\end{equation}
We start by separating the oscillations: 
\begin{equation*}
	S(N)    = \sum_{r\in\SO} \upvchi_{it,m}(r)  V\lp \frac{|r|}{N}\rp\sum_{n\in\SO} \lambda_\uppii(n) \delta(r = n).
\end{equation*}
Let $p $ be prime.  On applying the Bessel $\delta$-identity in Theorem \ref{6thm: Bessel identity}, we obtain
\begin{equation}\label{5eq: S(N)=Sp(N,X)}
S(N) = S_p^\star(N,X) + S_p^0(N,X)+ O(X^{-A}),
\end{equation}
with
\begin{equation*}
	\begin{split}
		S_p^\star(N,X) = \frac{|d_F| N^{1/2}}{\mathrm{N}(p)^{3/2}X^{3/2}} &\sum_{ r \in\SO} \upvchi_{it,m}(r)V^{}_{^\natural} \bigg(\frac{|r|}{N} \bigg) \sum_{a\in(\SO/p\SO)^\times} e_F\bigg[ \frac{ar}{p}\bigg]\\
		\cdot &\sum_{ n \in \SO} \lambda_\uppii(n) e_F\bigg[\! -\frac{an}{p}\bigg] I_{\vnu,q}\bigg(\frac{\sqrt r}{p},\frac{\sqrt n}{p};X\bigg),
	\end{split}
\end{equation*} 
\begin{equation*}
	S_p^0(N,X) = \frac{|d_F| N^{1/2}}{\mathrm{N}(p)^{3/2}X^{3/2}} \sum_{ r \in\SO} \upvchi_{it,m}(r)V^{}_{^\natural} \bigg( \frac{|r|}{N}\bigg) \sum_{n\in\SO} \lambda_\uppii(n) I_{\vnu,q} \bigg(\frac{\sqrt r}{p},\frac{\sqrt n}{p};X\bigg),
\end{equation*}
where $V^{}_{^\natural} (x) \in C_c^{\infty} [1,2]$ has bounds $ V^{(j)}_{^\natural} (x) \Lt_j 1 $; more explicitly, 
\begin{equation*}
	V^{}_{^\natural} (x) =  \frac {\sqrt{x}  V(x) } {|d_F|}   V_{\vnu,q}  \bigg(\frac {NXx} {|p|^2} \bigg) . 
\end{equation*} 
As  $ I_{\vnu,q} ( {\sqrt r}/{p}, {\sqrt n}/ {p};X ) $ (defined as in \eqref{3eq: definition of I-integral}) is a Hankel-transform integral,  if we apply  the Vorono\"i summation formula in Lemma \ref{2lem: Voronoi summation formula}  with modulus $c=p$ in the reserved direction to the inner $n$-sum in  $S_p^\star(N,X)$, then we obtain
\begin{equation}
	\begin{split}
		S_p^\star(N,X) = {}&  \frac{N^{1/2}}{\mathrm{N}(p)^{1/2}X^{3/2}}   \sum_{ r \in\SO} \upvchi_{it,m}(r)V^{}_{^\natural} \bigg(\frac{|r|}{N} \bigg) \\
		& \cdot   2 \mathrm{Re} \sum_{ n \in\SO} \lambda_\uppii(n) U \bigg(\frac {|n / d_F|} {X} \bigg) S_F (r, n; p)  e_F\bigg[  \frac{2\sqrt{nr}}{p}\bigg],
	\end{split}
\end{equation}
where $S_F (r, n; p)$ is the Kloosterman sum 
\begin{equation*}
S_F (r, n; p) = 	\sum_{a\in(\SO/p\SO)^\times} e_F \bigg[ \frac{ar + \widebar{a} n}{p}\bigg]. 
\end{equation*}
Note that $ S_F (r, n; p) $ is real-valued. Similarly, if the reversed Vorono\"i summation formula with modulus $c=1$ is applied  to $S_p^0(N,X)$, then  
\begin{equation} 
		S_p^0(N,X) \! = \!  \frac{\mathrm{N}(p)^{1/2}N^{1/2}}{X^{3/2}}   \! \sum_{ r \in\SO}\! \upvchi_{it,m}(r)V^{}_{^\natural} \bigg(\frac{|r|}{N}  \bigg) \!
		\cdot 2 \mathrm{Re} \! \sum_{ n \in\SO} \! \lambda_\uppii(n) U\bigg( \bigg| \frac {p^2 n} {X d_F} \bigg|  \bigg) e_F [  2\sqrt{nr} ]. 
\end{equation} 
This can be trivially estimated as 
\begin{equation}\label{5eq: bound for Sp0(N,X)}
	S_p^0(N,X) \Lt \frac{N^{5/2}X^{1/2}}{\mathrm{N}(p)^{3/2}}.
\end{equation}
Further, let us average $S_p^\star(N,X) $ over the primes $p$ with $ |p| $ in the dyadic segment $[P, 2P]$ and $\arg p $ in the arc $ [0, 2\pi / w_F)$ (for brevity, write $p \sim P$).  Let $P^\star$ denote the number of primes $p \sim P$, then, by the Prime Ideal Theorem of Landau, we have $ P^{\star} \asymp P^2 / \log P$. 

From \eqref{5eq: S(N)=Sp(N,X)}--\eqref{5eq: bound for Sp0(N,X)}, let us conclude this section with the following formula for $S (N)$. 

\begin{proposition}\label{6prop: S(N) = S(N,X,P)}
	 Let  $N,X,P$ be large parameters such that \begin{equation}\label{5eq: assumptions} 
	 	 P^2/N \Lt X, \qquad   N < X^{1-\vepsilon} .
	 \end{equation} Then 
	 \begin{equation}\label{6eq: S(N)=S(N,X,P)}
	 	S(N) = S^\star(N,X,P) + O\left(\frac{N^{5/2}X^{1/2}}{P^3} \right),
	 \end{equation}
	 where 
	 \begin{equation}\label{6eq: after average over p}
	 	\begin{split}
	 		S^\star(N,X,P) = {}& \frac{N^{1/2}}{P^\star X^{3/2}}    \sum_{p\sim P} \frac{1}{\mathrm{N}(p)^{1/2}} \sum_{  r \in\SO} \upvchi_{it,m}(r) V^{}_{^\natural} \bigg(\frac{|r|}{N}\bigg) \\
	 		& \cdot    2 \mathrm{Re} \sum_{ n \in\SO} \lambda_\uppii(n) U \bigg(\frac {|n / d_F|} {X} \bigg) S_F (r, n; p)  e_F\bigg[  \frac{2\sqrt{nr}}{p}\bigg].
	 	\end{split}
	 \end{equation}
\end{proposition}

	\section{Application of Poisson Summation and   Cauchy Inequality} 
	
	For brevity, let us subsequently write
	\begin{equation}
		C = \RC (t, m). 
	\end{equation}
Set
\begin{equation}\label{7eq: condition for X}
	X  = \frac { P^2K^2} {N}, \qquad N^{\vepsilon} < K < C^{1-\vepsilon},
\end{equation}
with the parameter $K$  to be optimized later. Then assumptions in \eqref{5eq: assumptions} amount to
\begin{equation}\label{7eq: condition for P}
	P K>N^{1+\vepsilon}.
\end{equation}

\subsection{Application of the First Poisson Summation} 

By \eqref{3eq: rho, theta}, write
\begin{align}
	\upvchi_{it,m}(r) = e(t\rho(r)+m\theta(r)), \qquad \rho(z) = \frac{\log|z|}{2\pi},\quad \theta(z) = \frac{\arg(z)}{2\pi}, 
\end{align}
and open the Kloosterman sum in the $r$-sum in \eqref{6eq: after average over p}, then we have 
\begin{equation*}
	2 \sum_{a\in(\SO/p\SO)^\times} e_F \bigg[ \frac{  \widebar{a} n}{p}\bigg] \sum_{r   \in \SO}  e_F \bigg[\frac {a r} {p} \bigg]    e ( t\rho(r)+m\theta(r) )  \mathrm{Re} \bigg\{ e_F  \bigg[\frac {2\sqrt{n r}} {  p} \bigg]  \bigg\} V^{}_{^\natural} \bigg(\frac{|r|}{N}  \bigg)  . 
\end{equation*} 
By the Poisson summation formula in Corollary \ref{cor: Poisson} with modulus $c = p$, this sum is transformed into 
\begin{equation}\label{7eq: after Poisson}
	\frac{N^{2+ it} }{\sqrt{|d_F|}}  \sum_{(r,p)=(1)} e_F\left[   \frac{\bar r n}{p}\right] \big(\text{\large $\EuScript{J}$}    (\sqrt{n}; r,p) + \text{\large $\EuScript{J}$}   (- \sqrt{n}; r,p) \big), 
\end{equation}
where the dual $r$-sum  is over $\SO$ by default, and, after the change $\pm \sqrt{ z / N} \ra z$,  the Fourier integral reads 
\begin{equation}\label{7eq: J(w,r,p)}
 	\text{\large $\EuScript{J}$}  (w; r,p) =  2 \! \viint V_{\scriptscriptstyle\natural}(|z|^2)  e \bigg( 2t\rho(z)+2m\theta(z) + \Tr \bigg(\frac{ 2 \sqrt{N} w z + Nrz^2 }{\sqrt{d_F} p  }    \bigg) \!  \bigg)\nd \mathrm{A} . 
\end{equation} 

	\begin{lem}\label{7lem: bound for J(w)}
	Let $|w| \asymp \sqrt{X} $. Then the integral $\text{\large $\EuScript{J}$} (w; r,p)$ is negligibly small unless $|Nr/p| \Lt C$, in which case 
	\begin{equation*}
		 \text{\large $\EuScript{J}$} (w;r,p) \Lt \frac 1 {C} . 
	\end{equation*} 
\end{lem}

\begin{proof}
	Apply Lemma \ref{4lem: pre-analysis of I} with $ \valpha = Nr / \sqrt{d_F} p$ and $ \delta (z) = 2 \sqrt{N} w z / \sqrt{d_F} p$. Note that the bound for $\delta (z)$ therein is due to $ \sqrt{N X} / P < C^{1-\vepsilon} $ as in \eqref{7eq: condition for X}. 
\end{proof}

By Lemma \ref{7lem: bound for J(w)}, at the cost of a negligible error, the $r$-sum  may be truncated at
\begin{equation}\label{7eq: definition of R}
	R = \frac {PC} N.
\end{equation}
For simplicity, let us write 
\begin{equation}\label{7eq: EuScript G}
	\text{\large $\EuScript{G}$}    (w; r,p)=\text{\large $\EuScript{J}$}    (\sqrt{w}; r,p) + \text{\large $\EuScript{J}$}   (-\sqrt{w}; r,p) .
\end{equation}
As $ X = { P^2K^2} / {N}$, if we omit the negligible error, then it follows from  \eqref{6eq: after average over p}, \eqref{7eq: after Poisson}, and \eqref{7eq: EuScript G} that 
\begin{align*} 
		S^\star(N,X,P) & ={}   \frac{  N^{4+it}}{P^\star P^3K^3\sqrt{|d_F|}} \! \sum_{ n } \lambda_\uppii(n) U \bigg(\frac {|n / d_F|} {X} \bigg) \! \sum_{p\sim P}\frac{1}{ |p | } \! \! \sum_{\substack{|r|\Lt R\\(r,p)=(1)}} \!  e_F\left[\frac{\bar r n}{p}\right]\text{\large $\EuScript{G}$}(n; r,p) .
\end{align*}

\subsection{Application of the Cauchy Inequality and the Second Poisson Summation}

Next, we apply the Cauchy inequality  and the Ramanujan bound on average for $\lambda_\uppii(n)$ as in \eqref{4eq: Ramanujan bound}, obtaining 
\begin{align*} 
	S^\star(N,X,P)    \Lt    \frac{  N^{3}}{P^\star P K } \Bigg(  \sum_{ n }   \Bigg| \sum_{p\sim P}\frac{1}{ |p | } \sum_{\substack{|r|\Lt R\\(r,p)=(1)}}  e_F\left[\frac{\bar r n}{p}\right]\text{\large $\EuScript{G}$}(n; r,p) \Bigg|^2 U \bigg(\frac {|n / d_F|} {X} \bigg)  \Bigg)^{1/2}. 
\end{align*}
By opening the square and interchanging the order of summations, we have 
\begin{equation}\label{7eq: after Cauchy-Schwarz}
\begin{split}
	S^\star(N,X,P)^2 \Lt &	\frac{  N^{6}}{(P^{\star } P  K)^2 }\hskip-1pt \underset{p_1,p_2\sim P}{\sum\sum}\frac{1}{|p_1p_2|}\underset{\substack{|r_1|, |r_2| \Lt R\\ (r_1,p_1) = (r_2, p_2)=(1)}}{\sum\sum} \\
& \cdot  \sum_{ n } e_F\left[\frac{\bar r_1 n}{p_1}-\frac{\bar r_2 n}{p_2}\right]\! \text{\large $\EuScript{G}$} (n; r_1,p_1)\overline{\text{\large $\EuScript{G}$} (n; r_2,p_2)} U \bigg(\frac {|n / d_F|} {X} \bigg).
\end{split}
\end{equation}
By  applying the Poisson summation formula  in Corollary \ref{cor: Poisson} with modulus $c=p_1 p_2$,   the $n$-sum in the second line of \eqref{7eq: after Cauchy-Schwarz} is transformed into 
\begin{equation}\label{7eq: the 2nd Poisson}
	{ X^2 } {\textstyle \sqrt{|d_F^{ 3}|}} \sum_{n\equiv\bar r_1 p_2-\bar r_2 p_1  (\mathrm{mod} \, p_1 p_2)}\CL  ( {X   \sqrt{d_F}n } / {p_1 p_2}; r_1,r_2,p_1,p_2  ), 
\end{equation}
where 
\begin{equation}\label{7eq: L(v)}
	\CL(v) \! = \! \CL(v;r_1,r_2,p_1,p_2) \! = \!\! \viint \! U(|z|)\text{\large $\EuScript{G}$} (Xd_Fz; r_1,p_1)\overline{\text{\large $\EuScript{G}$} (Xd_Fz; r_2,p_2)}e [vz] \nd \mathrm{A}.
\end{equation}

\section{\texorpdfstring{Analysis of the Integral $ \boldsymbol{\CL (v)}$}{Analysis of the Integral L(v)}}   

Trivially, Lemma \ref{7lem: bound for J(w)} yields the bound
\begin{equation}\label{7eq: trivial bound of L}
	\CL(v)\Lt \frac {1\,} {C^2}.
\end{equation}
However, in view of \eqref{7eq: condition for X}, \eqref{7eq: J(w,r,p)}, \eqref{7eq: EuScript G}, and \eqref{7eq: L(v)}, after the change $ \pm \sqrt{z} \rightarrow z$,  we may write 
	\begin{equation}\label{8eq: integral}
	\begin{split}
		\CL(v)  = &   \viint    \viint V^{}_{^\natural} ({\textstyle |z  _1|^2})\overline{V^{}_{^\natural} ({|z  _2|^2})}e \big(\psi(z  _1)-\psi(z  _2)+ \Tr(\valpha_1 z  _1^2-\valpha_2 z  _2^2 ) \big)  \\
		  & \cdot 2 \! \viint  |4 z|^2U(|z|^2)e [ v z^2- 2K w (z  _1,z  _2) z ] \nd \mathrm{A}(z) \, \nd \mathrm{A}(z  _1)\, \nd \mathrm{A}(z  _2) ,
	\end{split}
\end{equation} 
if we introduce 
\begin{equation}\label{8eq: notation}
	\psi (z) = 2t\rho(z)+2m\theta(z), \quad \valpha_1= \frac {Nr_1}  {\sqrt{d_F}p_1 }, \  \valpha_2 = \frac {Nr_2}  {\sqrt{d_F}p_2 }, \quad w(z  _1,z  _2)= P   \bigg(\frac {z  _1} {p_1} - \frac {z  _2} {p_2} \bigg). 
\end{equation} 
Now we use the lemmas from \S \ref{sec: analysis of integrals} to analyze the integrals in \eqref{8eq: integral}: the main results are summarized in the following lemma.  

\begin{lem}
	\label{8lem: integral}
	Let $N,K,P$ be parameters such that $N^{\vepsilon}<K<C^{1-\vepsilon}$. 
	Let $p_1, p_2 \sim P$ and $|r_1|, |r_2| \Lt PC/N$. 
	
\begin{itemize}
	\item [{\rm(1)}] $\CL(v)$ is negligibly small if $|v| \Gt K$. 
	\item [{\rm(2)}] Assume that $  K^2/C > N^{\vepsilon} $. Then for $N^{\vepsilon} K^2/C  \Lt |v| \Lt K$, we have 
	\begin{equation}\label{8eq: bound for L(v), main}
		\CL(v) \Lt \frac{1}{C^2 |v| } , 
	\end{equation}
and for $|v| \Lt N^{\vepsilon} K^2/C $, we have
	\begin{equation}\label{8eq: bound for L(v), 2}
		\CL(v) \Lt \frac{1}{C^2}.
	\end{equation}
\item [{\rm(3)}] If $p_1 = p_2 = p$, then $\CL(0) $ is negligibly small for $  |r_1 - r_2| > N^{\vepsilon} PC / KN  $. 
\end{itemize}
\end{lem}

\begin{proof}
	 First of all, as in Lemma \ref{4lem: Fourier integral}, the integral in the second line of \eqref{8eq: integral} is of the form $\text{\large $\EuScript{K}$}(v, Kw (z_1, z_2))$. Hence Lemma \ref{4lem: Fourier integral} (1) yields immediately the statement in {\rm(1)}, and, further, up to a negligible error,  the following reformulation of \eqref{8eq: integral}: 
	 \begin{equation*}
	 	\CL(v) = \frac{1}{|v|}  \viint \! \viint V^{}_{^\natural}({\textstyle |z_1|^2})\overline{V^{}_{^\natural}({ |z_2|^2})} W^{}_{^\natural}(K w (z_1, z_2)  /v; v) e( f  (z_1,z_2)) \nd \mathrm{A}(z_1)\, \nd \mathrm{A}(z_2),
	 \end{equation*} 
 with  the phase
 \begin{align*}
 	f  (z_1,z_2) = \psi(z  _1) - \psi(z  _2) + \Tr \big(\valpha_1 z_1^2 -\valpha_2 z_2^2  \big) - K^2 \cdot \Tr \bigg(\frac{w(z_1, z_2)^2}{v} \bigg), 
 \end{align*}
and the weight 
 \begin{align*}
 	W^{}_{^\natural}(u; v)  = |v| \cdot	\text{\large $\EuScript{W}$}(v u^2; v ) F (|u|), 
 \end{align*}
where $F \in C_c^{\infty} (\BR_+)$ is a suitable fixed cut-off function (dependent only on $U $). It suffices to know that  $ W^{}_{^\natural}(\cdot ; v)  $ is supported on a fixed annulus and that 
all its derivatives are uniformly bounded. Next, in order to separate variables, by Fourier inversion, we may write 
\begin{equation*}
	W^{}_{^\natural}(u; v) = \viint \hat{W}^{}_{^\natural}(z; v)e[-z u] \nd \mathrm{A} ,
\end{equation*}
where the Fourier transform exhibits rapid decay (uniformly in $v$):
\begin{equation*}
	\hat{W}^{}_{^\natural}(z; v) \Lt_{A} (1+|z|)^{-A}.
\end{equation*}
Thus, up to a negligible error, we arrive at 
\begin{equation*}
	\CL(v) = \frac{1}{|v|}\!  \viint_{D (N^\vepsilon)} \! \! \hat{W}^{}_{^\natural} (z; v) \!  \viint \! \viint \! V^{}_{^\natural}({\textstyle |z_1|^2}) \overline{V^{}_{^\natural}(|z_2|^2)} e(f(z_1,z_2; z, v)) \nd \mathrm{A}(z_1)\, \nd \mathrm{A}(z_2)\, \nd \mathrm{A}(z),
\end{equation*}
where as usual $D (r) \subset \BC$ is the disc of radius $r$ centered at the origin, and 
\begin{equation*}
	f (z_1,z_2 ; z, v) =  f (z_1,z_2) - K \cdot  \Tr\bigg(\frac{z \cdot w(z_1, z_2)   }{v}\bigg).
\end{equation*}
By applying the second derivative test in Lemma \ref{4lem: analysis of I, 2}, with 
\begin{equation*}
	 \delta(z_1,z_2; z, v)
	=
	- \frac{K^2  w(z_1,z_2)^2 + K z w(z_1,z_2) }{v}, 
\end{equation*}
we readily derive the bound in \eqref{8eq: bound for L(v), main}; here, by the definition of $w (z_1, z_2)  $ in \eqref{8eq: notation}, it is easy to verify that the derivatives of $ \delta (z_1, z_2; z, v) $ are all $ O (C/ N^{\vepsilon}) $ for any $|z| \leqslant N^\vepsilon$ and $|v|\Gt N^{\vepsilon} K^2/C $. Moreover, the bound in \eqref{8eq: bound for L(v), 2} has already been proven in \eqref{7eq: trivial bound of L}.

Finally, let us prove the assertion in (3) in the case that
\begin{align*}
	 \valpha_1= \frac {Nr_1}  {\sqrt{d_F}p }, \quad  \valpha_2 = \frac {Nr_2}  {\sqrt{d_F} p }, \qquad w(z  _1,z  _2)=   \frac {P(z_1 - z_2)} p.  
\end{align*}
To this end, let $w = w (z_1, z_2)$ be the new variable so that, according to Lemma \ref{4lem: Fourier integral} (2), the inner $z$-integral (see \eqref{8eq: integral}) is just $W_0( K w )$ for a radial Schwartz function $W_0$.  Now  we may write
\begin{equation*}
	\CL(0)=\viint W_0(K w)\viint V_0(w,z_2) e ( f_0(w,z_2)  )  \nd \mathrm{A}(z_2) \, \nd \mathrm{A}(w),
\end{equation*}
where 
\begin{align*}
	V_0(w,z_2) = |p/P|^2 V^{}_{^\natural}  \big( {\textstyle | pw/P+z_2 |^2} \big) \overline{V^{}_{^\natural}  \big(|z_2 |^2 \big) },
\end{align*}
\begin{align*}
	f_0(w,z_2) = \psi  (pw/P+z_2)  - \psi  (z_2)  + \Tr \big( \valpha_1   (pw/P+z_2) ^2-\valpha_2 z_2^2\big) .
\end{align*}
Let us restrict the $w$-integral to the small disc $D (N^{\vepsilon} / K)$, 
on which  $V_0(w, \cdot)$ is of uniformly bounded derivatives, while for $|r_1| \Lt PC / N$, we have
\begin{align*}
	\partial f_0 (w, z_2) / \partial z_2 & = \frac {2 N (r_1-r_2) z_2} {\sqrt{d_F} p}  + \frac{2 N r_1 w}{\sqrt{d_F} P} - \frac{(t-im) pw / P}{2\pi z_2 (   z_2 + pw / P )}  \\
	& = \frac {2 N (r_1-r_2) z_2} {\sqrt{d_F} p}  + O \bigg(\frac{CN^\vepsilon}{K} \bigg),
\end{align*}
 	It follows from $|r_1 - r_2| > N^{\vepsilon} PC / KN $ that $|	\partial f_0/\partial z_2|\Gt N|r_1-r_2|/P > N^{\vepsilon}$. Thus, for such $r_1, r_2$, we infer that $ \CL (0)$ is negligibly small by   Lemma \ref{4lem: Stationary Phase}.       \delete{, and, similarly, 
 \begin{align*}
 	\partial^2 f_0 (w, z_2) / \partial z_2^2 = \frac {2 N (r_1-r_2) } {\sqrt{d_F} p}  + O \bigg(\frac{CN^\vepsilon}{K} \bigg), 
 \end{align*}
 \begin{align*}
	\partial^j f_0 (w, z_2) / \partial z_2^j = O_{j} \bigg(\frac{CN^\vepsilon}{K} \bigg), \qquad \text{($j \geqslant 3$)}. 
\end{align*}}  
\end{proof}

	\section{\texorpdfstring{Estimates for $S^\star(N,X,P)$}{Estimates for S*(N,X,P)}}

First let us assume that  
\begin{equation}\label{9eq: condition on K}
	N, K > \sqrt{C} N^{\vepsilon} ,
\end{equation}
where $ K > \sqrt{C} N^{\vepsilon} $ is the condition  in Lemma \ref{8lem: integral} {\rm(2)}. 
By \eqref{7eq: after Cauchy-Schwarz}, \eqref{7eq: the 2nd Poisson}, along with Lemma \ref{8lem: integral}, we conclude that
\begin{equation}\label{9eq: after Cauchy-Schwarz}
	S^\star(N,X,P) \Lt \sqrt{ S_\text{diag}^2(N,X,P)} + \sqrt{  S_{\mathrm{off}}^2 (N,X,P)} + N^{-A},
\end{equation}
where
\begin{equation} \label{9eq: diag definition} 
	S_\text{diag}^2 (N,X,P)=\frac{N^6 X^2}{(P^{\star  } P^2  K)^2 }\sum_{p\sim P}\mathop{\sum}_{\substack{|r |\Lt R\\ (r  ,p)=(1)  }} \frac{1}{C^2}, 
\end{equation}
and 
\begin{equation}\label{9eq: off definition} 
	\begin{split}
		S_\text{off}^2 (N,X,P) & =  \frac{N^6 X^2}{(P^{\star  } P^2  K)^2 }   \mathop{\sum \sum}_{p_1,p_2 \sim P} \mathop{\sum \sum}_{\substack{|r_1|,|r_2|\Lt R\\ (r_1  ,p_1) = (r_2, p_2) =(1) }} \\ 
		& \cdot \left\{ \mathop{\sum}_{ \substack{N^{1+\vepsilon}/ C \Lt |n| \Lt N / K \\ n \equiv  \bar r_1   p_2 - \bar r_2   p_1 (\mathrm{mod}\, p_1 p_2) }  }    \frac{|p_1 p_2|}{ C^2 X |n| }  + \mathop{\sum}_{ \substack{ 0 < |n| \Lt N^{1+\vepsilon}/ C \\ n  \equiv  \bar r_1   p_2 - \bar r_2   p_1 (\mathrm{mod}\, p_1 p_2) }  } \frac 1 {C^2}  \right\},
	\end{split}
\end{equation}
in correspondence to the cases where $n = 0$ and $n \neq 0$ in \eqref{7eq: the 2nd Poisson}, respectively.  
Note that for $n=0$ the congruence condition in \eqref{7eq: the 2nd Poisson} reads $  r_1   p_1 \equiv   r_2   p_2 (\mathrm{mod}\, p_1 p_2)$ and this implies   $ p_1 = p_2  $($= p$) and $r_1 \equiv r_2 (\mathrm{mod}\, p)$ (keep in mind that $\arg p_1, \arg p_2 \in [0, 2\pi/w_F)$), but this in turn forces $r_1 = r_2$($= r$), since if otherwise we would have $ |r_1 - r_2| \geqslant |p |  > N^{\vepsilon} PC / KN  $ due to \eqref{9eq: condition on K} (see Lemma \ref{8lem: integral} (3)).  Moreover, when applying the estimates in Lemma \ref{8lem: integral} (2) with $v=X\sqrt{d_F}n/p_1p_2$, note that $N^\vepsilon K^2/C\Lt |v|\Lt K$ or $|v|\Lt N^\vepsilon K^2/C$ amounts to $N^{1+\vepsilon}/ C\Lt |n|\Lt N/K$ or $|n|\Lt N^{1+\vepsilon}/ C$, respectively, for $X=P^2K^2/N$ (see \eqref{7eq: condition for X}). 

From \eqref{9eq: diag definition}, we deduce that 
\begin{equation}\label{9eq: Sdiag bound}
	\sqrt{S_{\text{diag}}^2(N,X,P)} \Lt \frac{N^3 X }{P^{\star  } P^2  K }   \frac {\sqrt{P^{\star}} R } {C}   \Lt  KN   \sqrt{\log P}, 
\end{equation}
where we have used $NX=P^2K^2$ and $R=PC/N$ as in \eqref{7eq: condition for X} and \eqref{7eq: definition of R}.


As for the sum $S_\text{off}^2 (N,X,P)$ defined in \eqref{9eq: off definition}, we first observe that  necessarily  $ p_1 \neq p_2 $. 
Otherwise, if $p_1 = p_2 = p$, then the congruence $ n  \equiv  \bar r_1   p - \bar r_2   p \,  (\mathrm{mod}\, p^2)  $ would imply $p \,| \, n$. This is impossible in view of the assumption $    N^{1+\vepsilon}/ K < P$ in \eqref{7eq: condition for P} and the length $N/K $ of the $n$-sum. Now we interchange the sum over $n$ and the sums over $r_1,r_2$. For  $n$ given, the congruence $n  \equiv  \bar r_1   p_2 - \bar r_2   p_1 (\mathrm{mod}\, p_1 p_2)$ splits into $r_1\equiv \bar np_2 (\mathrm{mod}\, p_1)$ and $r_2\equiv -\bar np_1 (\mathrm{mod}\, p_2)$, so 
\begin{equation*}
	\begin{split}
		S_{\text{off}}^2 & (N,X,P) =\frac{N^6 X^2}{(P^\star P^2 K)^2}\mathop{\sum \sum}_{\substack{ p_1\neq p_2\\p_1, p_2 \sim P }}\\
		\cdot&\left\{\sum_{N^{1+\vepsilon}/ C \Lt |n|\Lt N/K} \hskip -6pt \mathop{\sum \sum}_{\substack{ |r_1|, |r_2| \Lt R\\ r_1 \equiv  \bar{n}p_2(\mathrm{mod}\, p_1)\\ r_2 \equiv - \bar{n}p_1(\mathrm{mod}\, p_2)}}  \!\! \frac{|p_1 p_2|}{C^2X|n|} + \sum_{0<|n|\Lt N^{1+\vepsilon}/ C}  \mathop{\sum \sum}_{\substack{ |r_1|, |r_2| \Lt R\\ r_1 \equiv  \bar{n}p_2(\mathrm{mod}\, p_1)\\ r_2 \equiv -  \bar{n}p_1(\mathrm{mod}\, p_2)}} \! \! \frac{1}{C^2}	\right\}.
	\end{split}
\end{equation*}
Consequently, for $N < C^{1+\vepsilon}$ as in Theorem \ref{thm: bound for S(N)}, so that $ R / P >  1 / C^{\vepsilon} $, we have 
\begin{equation}\label{9eq: Soff bound}
	\sqrt{S_{\text{off}}^2(N,X,P)} \Lt N^{\vepsilon} \frac{N^3 X }{P^{\star  } P^2  K }  P^{\star} \bigg(   \sqrt{\frac {P^2} {C^2 X} \frac {N} {K} }   +   \frac {N } {C^2}   \bigg) \! \lp \frac {R} {P} \rp^2 \!  = N^{\vepsilon} \bigg( \frac {C N} {\sqrt{K}} + K N \bigg) . 
\end{equation}
We conclude from \eqref{9eq: after Cauchy-Schwarz}, \eqref{9eq: Sdiag bound}, and \eqref{9eq: Soff bound} that 
\begin{equation}\label{9eq: S}
	S^\star(N,X,P) \Lt  N^{1+ \vepsilon } \bigg( \frac {C } {\sqrt{K}} + K   \bigg) .
\end{equation}

Finally, by \eqref{6eq: S(N)=S(N,X,P)}  in Proposition \ref{6prop: S(N) = S(N,X,P)} and the preceding estimate for $S^\star(N,X,P)$, we have 
\begin{equation}
	S (N) \Lt N^{1+ \vepsilon } \bigg( \frac {C } {\sqrt{K}} + K + \frac {K N} {P^2} \bigg) , 
\end{equation}
and hence, on the choices $K  = C^{2/3}$ and $P = N $, we obtain the bound \eqref{3eq: bound for S(N)} in Theorem \ref{thm: bound for S(N)}. Note that the required conditions in \eqref{7eq: condition for X}, \eqref{7eq: condition for P} and \eqref{9eq: condition on K} are well justified for our choices of $K$ and $P$.

	\def\cprime{$'$}


\end{document}